\documentclass[oneside, a4paper,12pt,reqno]{amsart}
\newtheorem{theorem}{Theorem}

\newtheorem{lemma}[theorem]{Lemma}

\begin{document}
\title{Random walk in the low disorder ballistic regime}

\author{A. F. Ram\'\i rez}

\address{Facultad de Matem\'aticas, Pontificia Universidad
Ca\'olica de Chile,\\
Santiago, Regi\'on Metropolitana 7820436, Chile\\
E-mail: aramirez@mat.uc.cl\\
www.uc.cl}

\begin{abstract}We consider a random walk in 
 $\mathbb Z^d$ which  jumps from a site $x$ to
  a nearest neighboring site $x+e$ (where $e\in V:=\{x\in\mathbb Z^d: |x|_1=1\}$) with 
probability $p_0(e)+\epsilon\xi(x,e)$. Here
$\sum_e p_0(e)=1$, $p_0(e)> 0$, $\epsilon$ is
a small parameter while $\{\{\xi (x,e):e\in V\}: x\in\mathbb Z^d\}$
are i.i.d. random variables with an absolute value bounded by $1$.
We review recent progress in the non-vanishing velocity case, giving an asymptotic expansion in
$\epsilon$ of the invariant measure of the environmental process,
and bounds for the velocity.
\end{abstract}

\keywords{Random walk in random environment; Ballisticity; Invariant probability measure; low disorder; asymptotic expansions.}

\subjclass[2010]{60K37, 82D30, 82C41.}


\maketitle

\section{Introduction}\label{aba:introduction}

Random perturbations of random walks have attracted
attention both from the mathematical and physical literature,
as a natural model of movement in a disordered media (see
for example the reviews Ref. \cite{dr14} and \cite{z06}).
Fundamental questions about its behavior remain open \cite{dr14,z06}.
In this article, we will review the actual state of understanding
of the so called ballistic case (non-vanishing velocity) 
at low disorder and will specially describe
recent results obtained providing asymptotic expansions for the velocity and the invariant
measure \cite{cr15,lrs16}.

For any $x\in\mathbb Z^d$, define as $|x|_1$ its $l_1$ norm.
Let $V:=\{e\in\mathbb Z^d:|e|_1=1\}$ and $\mathcal P:=\{p=\{p(e):e\in V\}:
p(e)\ge 0, \sum_{e\in V}p(e)=1\}$. We define
 $\Omega:=\mathcal P^{\mathbb Z^d}$ and call it the {\it environmental space}.
For each environment $\omega\in\Omega$ and $x\in\mathbb Z^d$, we define
the {\it random walk in the environment $\omega$} starting from $x$, as the
Markov chain with transition probabilities

\[
P_{x,\omega}(X_{n+1}=y+e|X_n=y)=\omega(y,e)\qquad {\rm for}\quad y\in
\mathbb Z^d, e\in V.
\]
Let $\mathbb P$ be a probability measure defined on $\Omega$,
  denote by $\mathbb E$ the corresponding
expectation associated to $\mathbb P$ and call $P_{x,\omega}$ the {\it quenched law}
of the random walk in random environment and refer to the semi-direct product $P_x:=\mathbb P\otimes P_{x,\omega}$ as
the {\it averaged} or {\it annealed} law of the random walk
starting from $x$.
Throughout this article, we will
assume that under $\mathbb P$ the random variables
$\{\omega(x):x\in\mathbb Z^d\}$
are i.i.d.

\medskip

We say that a random walk in random environment
is {\it ballistic} if $\mathbb P_0$-a.s. we have that

\begin{equation}
\label{ball}
\liminf_{n\to\infty}\frac{X_n}{n}>0.
\end{equation}
It is well known (see for example \cite{dr14}) that whenever
a random walk in random environment is ballistic, the {\it velocity}

\[
v:=\lim_{n\to\infty}\frac{X_n}{n}
\]
exists.

We will consider perturbations of the dynamics of a 
random walk in a deterministic environment with jump rates
$p_0\in\mathcal P$ at every site, where

\begin{equation}
\label{pnot}
\mathcal P_0:=\{p\in\mathcal P:\min_{e\in V}p(e)>0\}.
\end{equation}
Given $\epsilon>0$, 
define

\[
\Omega_{p_0,\epsilon}:=\left\{\omega\in\Omega:\sup_{x\in\mathbb Z^d, e\in V}|\omega(x,e)-p_0(e)|\le\epsilon\right\}.
\]
Furthermore, we will assume that $\epsilon>\min\{p_0(e):e\in V\}$ and define the constant

\[
\kappa:=\min\{p_0(e):e\in V\}-\epsilon,
\]
which is a lower bound for the jump probabilities of environments
$\omega\in\Omega_{p_0,\epsilon}$, so that $\omega(x,e)\ge\kappa>0$ for
every $x\in\mathbb Z^d$ and $e\in V$.

 An important quantity which will somehow give us
a control on how close we are to a possible regime of vanishing
velocity is the {\it local drift} of the random walk, defined
on an environment $\omega$ and at a point $x\in\mathbb Z^d$ as

\begin{equation}
\label{ld}
d(x,\omega):=\sum_{e\in V}e\omega(x,e).
\end{equation}
Throughout this article, we will consider three different conditions on the local drift.

\medskip

\noindent {\bf Linear local drift condition} (LLD). {\it Let $C>0$, $\epsilon>0$ and $p_0\in
\mathcal P_0$. We say that
the environmental law $\mathbb P$ satisfies the quadratic local drift condition (LLD)$_\epsilon$
if $\mathbb P(\Omega_{p_0,\epsilon})=1$
and

\begin{equation}
\label{LLD}
\mathbb E\left[d(0,\omega)\cdot e_1\right]> C\epsilon.
\end{equation}
}
\medskip

\noindent {\bf Quadratic local drift condition} (QLD). {\it Let $C>0$, $\epsilon>0$ and $p_0\in
\mathcal P_0$. We say that
the environmental law $\mathbb P$ satisfies the quadratic local drift condition (QLD)$_\epsilon$
if $\mathbb P(\Omega_{p_0,\epsilon})=1$
and

\begin{equation}
\label{QLD}
\mathbb E\left[d(0,\omega)\cdot e_1\right]> C\epsilon^2.
\end{equation}
}
\medskip

\noindent For the particular case of perturbations around the simple symmetric random walk, we will
introduce the following.

\medskip

\noindent {\bf Local drift condition} (LD). {\it Let $\eta,\epsilon>0$ and $p_0=\frac{1}{2d}$. We say that
an environment law $\mathbb P$ satsifies the local drift condition (LD)  if $\mathbb P(\Omega_{p_0,\epsilon})=1$
and

\begin{equation}
\label{LD}
\mathbb E\left[d(0,\omega)\cdot e_1\right]> C\epsilon^{\alpha(d)-\eta},
\end{equation}
where

\[
\alpha(d):=
\begin{cases}
2 &\ {\rm if}\ d=2\\
2.5 &\ {\rm if}\ d=3\\
3 &\ {\rm if}\ d\ge 4.
\end{cases}
\]

}
\medskip

Given an environment in $\omega\in\Omega_{p_0,\epsilon}$, we
will define

\[
\xi(x,e):=\frac{1}{\epsilon}\left(\omega(x,e)-p_0(e)\right),
\quad{\rm for}\quad x\in\mathbb Z^d, e\in V,
\]
and use the more suggestive writings

\begin{equation}
\label{suggestive2}
\omega(x,e)=p_0(e)+\epsilon\xi(x,e)\quad {\rm and}\quad
\omega(x,e)=p_\epsilon(e)+\epsilon\bar\xi(x,e),
\end{equation}
where $p_\epsilon(e):=p_0(e)+\mathbb E[\xi(0,e)]$ and $\bar\xi(x,e):=\xi(x,e)-p_1(e)$.

\medskip

\section{Invariant measures and a  heuristic derivation
of their expansion}
A fundamental and useful concept in the understanding of the
behavior of a random walk in random environment is
the invariant measure of the environment seen from the
position of the random walk \cite{dr14}. Here we will show
how to formally derive an expansion in the parameter
$\epsilon$ of it \cite{cr15}.

For each $x\in\mathbb Z^d$, we define the  shift  $t_x:\Omega\to\Omega$
as $t_x\omega(y):=\omega(x+y)$ for each $y\in\mathbb Z^d$.
Now, given an environment $\omega\in\Omega$, we define
the {\it environment seen from the random walk}
or the {\it environmental process} as

\[
\bar\omega_n:=t_{X_n}\omega\quad {\rm for}\quad n\ge 0.
\]
Note that $\{\bar\omega_n:n\ge 0\}$ is a Markov process
with state space $\Omega$ and transition kernel
defined for $f:\Omega\to\mathbb R$ by

\[
Rf(\omega):=\sum_{e\in V}\omega(0,e)f(t_e\omega).
\]
We say that a probability measure $\mathbb Q$ defined in
$\Omega$ is an {\it invariant} measure for the environmental process if for every bounded measurable function
$f:\Omega\to\mathbb R$ we have that
$\int Rfd\mathbb Q=\int fd\mathbb Q$.
Now, in the case of an environment $\omega\in\Omega_{p_0,\epsilon}$,
by (\ref{suggestive2}), we can write the transition kernel
as

\begin{equation}
\label{ere}
R=R_0+\epsilon A,
\end{equation}
where for $f:\Omega\to\Omega$ measurable we have

\[
R_0f(\omega)=\sum_{e\in V}p_\epsilon(e)f(t_e\omega)\quad{\rm and}\quad
Af(\omega)=\sum_{e\in V}\bar\xi(0,\omega)f(t_e\omega).
\]
Assume now that the measure $\mathbb Q$ is absolutely continuous
with respect to $\mathbb P$ with a Radon-Nykodim derivative $h$.
We then have that

\begin{equation}
\label{inv11}
\int R fhd\mathbb P=\int fhd\mathbb P.
\end{equation}
If we furthermore assume that $h$ has the analytic expansion

\begin{equation}
\label{expansiong}
h=\sum_{i=0}^\infty\epsilon^ih_i,
\end{equation}
we see inserting (\ref{ere}) and (\ref{expansiong})  into (\ref{inv11})
 that necessarily $h_0=1$ and that for each $i\ge 0$ one should have that

\begin{equation}
\label{gek}
h_{i+1}=-(R_0^*-I)^{-1}A^*h_i,
\end{equation}
where $R_0^*$ and $A^*$ denote the adjoint of $R_0$ and $A$
respectively with respect to the measure $\mathbb P$. Now,

\[
(R_0^*-I)^{-1}f(\omega)=-\sum_{z\in\mathbb Z^d} G^{p_\epsilon}(z,0)f(t_z\omega)
\quad{\rm and}\quad
A^*f(\omega)=\sum_{e\in V}\bar\xi(-e,e)f(t_e\omega),
\]
where for each $x,y\in\mathbb Z^d$ we define the Green function
$G^{p_\epsilon}(x,y)$ as the expected number of visits to site $y$ starting from site $x$
of a random walk with transition kernel $p_\epsilon$.
It follows that  
\(
h_1=\sum_{z\in\mathbb Z^d, e\in V}\bar\xi(z,e)G^{p_\epsilon}(0,z+e),
\)
and hence

\begin{equation}
\label{fe}
h=1+\epsilon\sum_{z\in\mathbb Z^d, e\in V}\bar\xi(z,e)G^{p_\epsilon}(z+e,0)
+O(\epsilon^2).
\end{equation}
Of course, for the above formal expansion to make sense, we will
need to require certain conditions on the random walk which
will ensure the existence of a non-vanishing speed.

\medskip

\section{Asymptotic expansion of the invariant measure}
Here we will show in what sense is the formal
expansion derived in the previous section valid.
Although it does not seem possible in general to
obtain an analytic expansion of the Radon-Nykodim
derivative of the invariant measure of the environmental
process as given by (\ref{expansiong}) and (\ref{gek}), 
 under
the lineal local drift condition  (LLD), it is possible
to derive an asymptotic expansion at least up to
first order.
Given a $B\subset\mathbb Z^d$ and a probability
measure $\mathbb S$ in $\Omega$, we define
$\mathbb S_B$ as the restriction of $\mathbb S$ to
$\mathcal P^B$, and with a slight abuse of
language we will call it just
the {\it restriction of $\mathbb S$ to $B$}.
Furthermore, we define for each $p\in\mathcal P$, $x\in\mathbb Z^d$ and
$e\in V$,

\[
J_{p}(x):=\lim_{n\to\infty}\sum_{k=0}^n(p_k(0,-x)-p(0,0)),
\]
where for each $n\ge 0$ and $x,y\in\mathbb Z^d$, we define
$p_n(x,y)$ as the probability that a random walk with transition
kernel $p$ jumps from $x$ to $y$ after $n$ steps.
The following theorem has been proved in Ref. \cite{cr15}.

\smallskip

\begin{theorem}[Campos-Ram\'\i rez]
\label{theorem1}
 Let $\eta>0$ and $B$  finite subset of $\mathbb Z^d$.
Then, there is an $\epsilon_0>0$
such that whenever $\epsilon\le \epsilon_0$,
 $p_0\in\mathcal P_0$, and $\mathbb P$ satisfies the
lineal local drift condition (LD) [c.f. (\ref{LLD})],
the limiting invariant measure $\mathbb Q$ has
a restriction $\mathbb Q_B$ to  $B$ which 
is absolutely continuous with respect to the restriction
$\mathbb P_B$ to $B$ of $\mathbb P$, with a
Radon-Nikodym derivative admitting $\mathbb P$-a.s. the expansion

\begin{equation}
\label{expansion1}
\frac{d\mathbb Q_B}{d\mathbb P_B}=1+\epsilon\sum_{z\in B}\sum_{e\in V}
\bar\xi(z,e) J_{p^*_\epsilon}(e+z)+ O\left(\epsilon^{2-\eta}\right),
\end{equation}
 where $\left|O\left(\epsilon^{2-\eta}\right)\right|\le c_1\epsilon^{2-\eta}$, for
some constant $c_1=c_1(\eta,\kappa,d,B)$ depending only on $\eta$, $\kappa$, $d$ and $B$.
Here $p^*_\epsilon(e):=p_\epsilon(-e)$ for each $e\in V$.
\end{theorem}
\smallskip

\noindent As a corollary of Theorem \ref{theorem1}, it is
possible to rederive the asymptotic expansion of the
velocity already obtained by Sabot \cite{sa04} under the local drift
condition  (LD). Indeed, this
condition implies that the random walk is ballistic (see
Ref. \cite{sa04}) and hence by an argument presented by
Sznitman and Zerner in Ref. \cite{SZ99} that
the marginal law of the environmental process at time $n$ converges
to the invariant measure as $n\to\infty$ of Theorem \ref{theorem1}.
On the other hand since
$X_n-\sum_{i=0}^{n-1}d(0,\bar\omega_i),\ i\ge 0$,
is a martingale, 
we can deduce from Theorem \ref{theorem1}, the asymptotic expansion for the velocity \cite{sa04}

\begin{equation}
\label{vexp}
v=\int d(0,\omega) d\mathbb Q=d_0+\epsilon d_1+\epsilon^2 d_2^\epsilon+ O\left(\epsilon^{3-\eta}\right),
\end{equation}
where 

\begin{equation}
\label{d0}
d_0:=\sum_{e\in V}ep_0(e),\quad
d_1:=\sum_{e\in V}e\mathbb E[\xi(0,e)],
\end{equation}

\begin{equation}
\label{d2eps}
d_2^\epsilon:=\sum_{e\in V}ep_{2,\epsilon}\quad {\rm and}\quad
p_{2,\epsilon}:=\sum_{e,e'\in V}C_{e,e'}J_{p_\epsilon^*}(e)
\end{equation}
with $C_{e,e'}:=Cov(\xi(0,e),\xi(0,e'))$, while
$|O(\epsilon^{3-\eta})|_1\le c_1'\epsilon^{3-\eta}$ for some
constant $c_1'$. 

In dimensions $d\ge 2$, it is possible to expand $J_{p_\epsilon^*}$ in $\epsilon$
to rewrite the expansion (\ref{expansion1}) as

\begin{equation}
\label{expansion2}
\frac{d\mathbb Q_B}{d\mathbb P_B}=1+\epsilon\sum_{z\in B}\sum_{e\in V}
\bar\xi(z,e) J_{p^*_0}(e+z)+ O\left(\epsilon^{2-\eta}\right),
\end{equation}
 where $\left|O\left(\epsilon^{2-\eta}\right)\right|\le c_2\epsilon^{2-\eta}$, for
some constant $c_2=c_2(\eta,\kappa,d,B)$ depending only on $\eta$, $\kappa$, $d$ and $B$.
Note that in the particular case when $p_0$ are the jump probabilities of a simple
symmetric random walk and the dimension $d=2$, as shown in
Ref. \cite{cr15}, we can use explicit
expressions for the potential kernel of the random walk (see Ref. 
\cite{McW40,sp64}), to obtain for example the following explicit expansion from
 (\ref{expansion2}),

\[
\frac{d\mathbb Q_{z_0,z_1}}{d\mathbb P_{z_0,z_1}}=1-
\frac{4}{\pi}\left(\bar\xi(z_1,e_1)+\bar\xi(z_1,-e_1)\right)\epsilon
+\left(\frac{8}{\pi}-4\right)\bar\xi(z_1,e_2)\epsilon+O\left(\epsilon^{2-\eta}
\right),
\]
where $z_0:=(0,0)$ and $z_1:=(1,0)$.

To prove Theorem \ref{theorem1} in Ref. \cite{cr15}, we
first express the
  limiting invariant measure as a Ces\`aro average
up to a random time distributed as a geometric random variable.
Indeed, let $\delta\in (0,1)$, and $\tau_\delta$ be a 
random time with a geometric distribution of
parameter $1-\delta$, independent of the random walk
and of the environment. Define for $x,y\in\mathbb Z^d$
the Green function

\[
g^\omega_\delta(x,y):=E'_{x,\omega}\left[\sum_{n=0}^{\tau_\delta} 1_y(X_n)\right],
\]
where the expectation is taken both over $\tau_\delta$ and
the random walk. We now define the probability measure $\mu_\delta$
by the equality valid for every continuous function $f:\Omega\to\mathbb R$
as

\begin{equation}
\label{ftx}
\int fd\mu_\delta=\frac{\sum_{x\in\mathbb Z^d}\mathbb E[g_\delta^\omega(0,x)
f(t_x\omega)]}{\sum_{x\in\mathbb Z^d}\mathbb E[g_\delta^\omega(0,x)]}.
\end{equation}
It turns out that whenever the lineal  local drift condition
 (LLD) is satisfied, the limit

\begin{equation}
\label{repr}
\mu:=\lim_{\delta\to 1^-}\mu_\delta
\end{equation}
exists weakly and is the limiting invariant measure for the environmental
process. In fact, this is true only under the assumption
that the polynomial ballisticity condition $(P)_M$ for
$M\ge 15d+5$ is satisfied \cite{cr15}. In section \ref{tvsdbr}
we will give a precise definition of the polynomial ballisticity condition.
 The representation of
the limiting invariant measure given by (\ref{repr}) and (\ref{ftx})
is then used in Ref. \cite{cr15} to derive the asymptotic
expansion of Theorem \ref{theorem1}, through the 
use of Green function expansions. On the other hand, this
method produces badly converging series when the perturbation
is around a $p_0\in\mathcal P_0$
[c.f. (\ref{pnot})] which has a vanishing
local drift.

\medskip

\section{The very small drift ballistic regime} 
\label{tvsdbr}
It is natural
to wonder, up to which point can the lineal local drift condition  (LLD)
be relaxed, and still obtain the asymptotic
expansion of the invariant measure given by Theorem \ref{theorem1}
and of the velocity of Sabot \cite{sa04}. It is still an
open problem to give a complete answer to this question.
We can nevertheless easily prove that
 under the weaker quadratic local drift condition (QLD), one still has a ballistic behavior
whenever the unperturbed random walk has a vanishing velocity.

\medskip

\begin{lemma}
\label{lemma1} Let $d\ge 2$ and $\eta\in (0,1)$. Let $\epsilon\in (0,1)$
and consider a random walk
whose law satisfies
 the quadratic local drift condition  (QLD)
with constant $C:=2/\min_{e\in V}p_0(e)^2$.
Then  the random walk is ballistic.
\end{lemma}
\begin{proof}
 We use the notation $d=d(0,\omega)$ for
the local drift [c.f. (\ref{ld})]. 
 Let $G$ be the set of vectors
$g=\{g(e):e\in V\}$ such that each component is in $[0,1]$.
To prove that a random walk is ballistic, it is
enough to show that the so called {\it Kalikow's criterion} is
satisfied (see Ref. \cite{sz02,k81}):

\[
\inf_{g\in G}\mathbb E\left[\frac{d(0,\omega)\cdot e_1}{\sum_{e\in V}
\omega(0,e)g(e)}\right]>0.
\]
Now, for each $g\in G$
one has that if $\epsilon$ is small enough

\begin{eqnarray*}
&E\left[\frac{d\cdot e_1}{
\sum_{e\in V}\omega(0,e)g(e)}\right]
=
E\left[\frac{(d\cdot e_1)_+}{
\sum_{e\in V}\omega(0,e)g(e)}\right]-
E\left[\frac{(d\cdot e_1)_-}{
\sum_{e\in V}\omega(0,e)g(e)}\right]\\
&
\ge E\left[\frac{(d\cdot e_1)_+}{
\sum_{e\in V}p_0(e)g(e)+\epsilon
\sum_{e\in V}g(e)}\right]-
E\left[\frac{(d\cdot e_1)_-}{
\sum_{e\in V}p_0(e)g(e)-\epsilon
\sum_{e\in V}g(e)}\right]\\
&\ge
\frac{\lambda}{\sum_{e\in V}p_0(e)g(e)}-
\frac{\epsilon\sum_{e\in V}g(e)}{\left(\sum_{e\in V}p_0(e)g(e)\right)^2}\mathbb E[(d\cdot e_1)_++
(d\cdot e_1)_-]\\
&\ge
\frac{1}{\sum_{e\in V}g(e)}\left( \lambda -\epsilon\frac{\mathbb E[(d\cdot e_1)_++
(d\cdot e_1)_-]
}{\min_{e\in V}p_0(e)^2}
\right)
>\frac{1}{2d}\left(\lambda-\epsilon^2\frac{2}{\min_{e\in V}p_0(e)^2}\right)>0,
\end{eqnarray*}
where we have used that the last two inequalities are clearly satisfied if $\lambda> C\epsilon^2$
where $\lambda:=\mathbb E[d(0,\omega)\cdot e_1]$.
\end{proof}

\medskip

Let us now try to explore if in the case of perturbations of
random walks with vanishing velocity, it would be possible to
improve the bound of Lemma \ref{lemma1} under a local drift condition weaker than the
quadratic one. We will  look at the velocity expansion (\ref{vexp}) for the case
in which the perturbation is performed on a simple symmetric
random walk, so that $p_0(e)=1/2d$ for all $e\in V$.
 To simplify notation we define $J_0:=J_{p_0}$. Note that  the factor $d_2^\epsilon=\sum_{e\in V}ep_{2,\epsilon}$
[c.f. (\ref{d2eps})]
in the term of second order of (\ref{vexp}) can be expanded as

\[
p_{2,\epsilon}=\sum_{e,e'\in V}C_{e,e'}J_0(e)+O(\epsilon)=O(\epsilon),
\]
where we have used the isotropy of the Green function which implies that
$J_0(e)$ is independent of $e\in V$. Furthermore, since $d_0=0$ [c.f. (\ref{d0})] for a
simple symmetric random walk, it follows by the expansion (\ref{vexp}), that in this case
we have that

\begin{equation}
\label{vede}
v=\epsilon d_1+O(\epsilon^{3-\eta})=\mathbb E[d(0,\omega)]+O(\epsilon^{3-\eta}).
\end{equation}
From (\ref{vede}) we see that at least formally, the velocity
does not vanish as long as we have

\begin{equation}
\label{ed0}
\mathbb E[d(0,\omega)]\cdot e_1> \epsilon^{3-\eta},
\end{equation}
for some $\eta>0$, which is the local drift condition (LD) [c.f. (\ref{LLD})] for dimensions $d\ge 4$.
 The following theorem shows that at least in dimensions $d\ge 4$
  condition (\ref{ed0}) does imply that the random walk
is ballistic [c.f. (\ref{ball})]. 

\medskip

\begin{theorem}[Sznitman]
\label{theorem2} Let $d\ge 2$ and $\eta\in (0,1)$. There
exists an $\epsilon_0\in (0,1)$ such that for
every $\epsilon\in (0,\epsilon_0)$,
whenever a random walk 
which is a random perturbation of a simple symmetric random walk
has a law which satisfies
 the local drift condition  (LLD), it
is ballistic.
\end{theorem}

\medskip

\noindent
The case $d=2$ of Theorem \ref{theorem2} is contained in Lemma \ref{lemma1}
above. In dimensions $d\ge 3$, Theorem \ref{theorem2} was proved by Sznitman in Ref. \cite{sz03} and includes cases of random walks
in random environments which do not satisfy Kalikow's condition
\cite{k81,SZ99}.
Heuristically, its proof is based on showing that at scales of order $\epsilon^{-1}$ the
behavior of the random walk is not far in terms of exit times from slabs from the behavior
of a simple symmetric random walk. Then, through a renormalization type argument
it is possible to derive a Solomon type criterion (originally for $d=1$ \cite{so75}), called the
{\it effective criterion} \cite{sz02} for slabs of
order $\epsilon^{-4}$, which then implies ballisticity.

\medskip

\section{Velocity estimates on the very small drift
ballistic regime}
In view of Theorem \ref{theorem2}, it is natural to ask wether or not the velocity
expansion (\ref{vexp}) is still valid under the local drift condition  (LD)  for $d\ge 3$. 
 The following result proved by Laurent, Ram\'\i rez, Sabot and Saglietti in Ref. \cite{lrs16}, shows that as an upper bound, the expansion
(\ref{vexp}) is still valid for perturbations around the simple symmetric random walk.

\medskip

\begin{theorem}[Laurent-Ram\'\i rez-Sabot-Saglietti]
\label{theorem3} Let $d\ge 3$. Consider a random walk in random environment satisfying 
 the local drift condition  (LD). Then, for
every $\eta>0$ there is a constant $C_\eta>0$ such that

\[
0<v\cdot e_1\le \mathbb E[d(0,\omega)]\cdot e_1+ C_\eta
\epsilon^{\alpha(d)-\eta}.
\]

\end{theorem}

\medskip

\noindent The proof of Theorem \ref{theorem3} 
 combines elements of the proof of Theorem \ref{theorem2}, with
renormalization type arguments similar in spirit to 
the one used to establish ballisticity of random walks
in random environments satisfying the  polynomial
ballisticity condition introduced in Ref. \cite{bdr14} and which
we define below.

Let $l\in\mathbb S^{d-1}$, $L>0$ and consider the box

\[
B_{l,L}:=R\left((-L,L)\times\left(-70 L^3,70L^3\right)^{d-1}\right)\cap
\mathbb Z^d
\]
where $R$ is a rotation in $\mathbb R^d$ fixing the origin and
defined by $R(e_1)=l$. Given $M\ge 1$ and $L\ge 2$, we say
that the {\it polynomial condition $(P)_M$
in direction $l$ is satisfied on a box of size $L$}
(also written $(P)_M|l$) if

\begin{equation}
\label{last}
P_0(X_{T_{B_{l,L}}}\cdot l<L)\le \frac{1}{L^M}.
\end{equation}
This condition was introduced Ref.  \cite{bdr14}, showing that
whenever the environment is i.i.d. and uniformly elliptic
and the polynomial condition is satisfied
 with $M\ge 15d+5$ on a box of size $L\ge c_0$ with

\[
c_0:=\frac{2}{3}2^{3(d-1)}\land\exp\left\{2\left(\ln 90+\sum_{j=1}^\infty
\frac{\ln j}{2^j}\right)\right\},
\]
 the random walk is necesarilly ballistic
and the decay in (\ref{last}) is actually 
at least $e^{-L^\gamma}$ for all $\gamma\in (0,1)$. The proof of this statement
involves a renormalization argument, where boxes are classified 
as {\it bad} or {\it good} according to the size of
the quenched probability that the random walk exits the box through
sides whose normal direction is not in the hyperspace
defined by $l$. A modification of this argument which keeps
track of the time spent on each box by the random walk, gives
the velocity estimate of Theorem \ref{theorem3}.

\section*{Acknowledgments} Supported by
Fondo Nacional de Desarrollo Cient\'\i fico y Tecnol\'ogico
grant 1141094 and Iniciativa Cient\'\i fica Milenio NC120062,

\end{document}